\documentclass{amsproc}[11pt]

\usepackage[ansinew]{inputenc}
\usepackage{graphicx}
\usepackage{color}

\definecolor{rr}{rgb}{.8,0,.3}

\usepackage{enumerate,latexsym}
\usepackage{latexsym}
\usepackage{amsmath,amssymb}
\usepackage{graphicx}
\newfont{\bb}{msbm10 at 11pt}
\newfont{\bbsmall}{msbm8 at 8pt}

\def\R{\mathbb{R}}

\def\C{\mathbb{C}}

\def\D{\mathbb{D}}

\def\esf{\mathbb{S}}

\newcommand{\ben}{\begin{enumerate}}
\newcommand{\bit}{\begin{itemize}}
\newcommand{\een}{\end{enumerate}}
\newcommand{\eit}{\end{itemize}}

\newcommand{\wt}{\widetilde}

\newcommand{\sol}{\mathrm{Sol}_3}
\newcommand{\ed}{\end{document}}

\def\cA{\mathcal{A}}
\def\cU{\mathcal{U}}

\def\cB{\mathcal{B}}

\def\cH{\mathcal{H}}
\def\cL{\mathcal{L}}
\def\cM{\mathcal{M}}

\let\8=\infty \let\0=\emptyset

\let\landa=\lambda
\let\alfa=\alpha

\let\parc=\partial

\def\landa{\lambda}

\def\flecha{\rightarrow}
\def\esiz{\langle}
\def\esde{\rangle}

\def\cte.{\mathop{\rm cte.}\nolimits}

 \def\dim{\mathop{\rm dim }\nolimits}

\def\E{\mathbb{E}}

\def\R{\mathbb{R}}

\def\C{\mathbb{C}}

\def\D{\mathbb{D}}

\def\H{\mathbb{H}}
\def\S{\mathbb{S}}

\def\Ek{\mathbb{E}^3 (\kappa,\tau)}

\newtheorem{theorem}{Theorem}[section]
\newtheorem{lemma}[theorem]{Lemma}

\newtheorem{remark}[theorem]{Remark}
\newtheorem{corollary}[theorem]{Corollary}
\newtheorem{definition}[theorem]{Definition}
\newtheorem{conjecture}[theorem]{Conjecture}

\newtheorem{example}[theorem]{Example}

\newcommand{\su}{{\rm SU}(2)}
\newcommand{\EE}{\wt{\mathrm E}(2)}

\renewcommand{\sl}{\wt{\mathrm{SL}}(2,\R)}

\numberwithin{equation}{section}

\begin{document}

\begin{title}
[A Hopf theorem and a conjecture of Alexandrov]{A Hopf theorem for non-constant mean curvature \\ and a conjecture of A.D. Alexandrov}
\end{title}
\author{José A. Gálvez}
\address{José A. Gálvez, Departamento de Geometría y Topología,
Universidad de Granada, 18071 Granada, Spain}
 \email{jagalvez@ugr.es}

\author{Pablo Mira}
\address{Pablo Mira, Departamento de Matemática Aplicada y Estadística, Universidad Politécnica de
Cartagena, 30203 Cartagena, Murcia, Spain.}

\email{pablo.mira@upct.es}

\thanks{The authors were partially supported by
MICINN-FEDER, Grant No. MTM2013-43970-P, Junta de Andalucía Grant No.
FQM325, and
Junta de Andalucía, reference P06-FQM-01642.}




\begin{abstract}
We prove a uniqueness theorem for immersed spheres of prescribed (non-constant) mean
curvature in homogeneous three-manifolds. In particular, this uniqueness theorem proves a conjecture by A.D. Alexandrov about immersed spheres of prescribed Weingarten curvature in $\R^3$ for the special but important case of prescribed mean curvature. As a consequence, we extend the classical Hopf uniqueness theorem for constant mean curvature spheres to the case of immersed spheres of prescribed antipodally symmetric mean curvature in $\R^3$.
\end{abstract}
\maketitle

\section{Introduction}
In his famous 1956 paper, A.D. Alexandrov \cite{A1} conjectured the following result:

\begin{conjecture}
Let $S\subset \R^3$ be a strictly convex sphere, and let $\Phi(k_1,k_2,x)\in C^1(\Omega)$ be a function such that $$\frac{\parc \Phi}{\parc k_1} \frac{\parc \Phi}{\parc k_2} >0$$ on the domain $\Omega\subset \R^2\times \S^2$ given by $$\Omega =\{(\landa \kappa_1(p),\landa \kappa_2(p),\nu(p))\in \R^2\times \S^2 : p\in S, \landa \in \R\},$$ where $\nu:S\flecha \S^2$ is the Gauss map of $S$, and $\kappa_1,\kappa_2:S\flecha \R$ are its principal curvatures.

Let $f:\S^2\flecha \R$ be the function defined by
\begin{equation}\label{eqfil}
\Phi(\kappa_1(p),\kappa_2(p),\nu(p)) = f(\nu(p)) \hspace{1cm} \forall p\in S.
 \end{equation}

Then any other compact surface $\Sigma$ of genus zero immersed in $\R^3$ whose Gauss map $\nu$ and principal curvatures $\kappa_1,\kappa_2$ satisfy \eqref{eqfil} is a translation of $S$.
\end{conjecture}

This conjecture is known to hold when $\Sigma$ is also a strictly convex sphere (\cite{A0,A1,P1,HW}; see \cite{GWZ} for a historical account of the problem and a generalization). In particular, this provides uniqueness for geometric problems formulated in terms of the \emph{curvature radii} $1/\kappa_i$ such as the Christoffel-Minkowski problem in $\R^3$, as the solutions to these problems are automatically strictly convex. Alexandrov stated without proof in \cite{A1} that the conjecture holds provided both $\Sigma,S$ are real analytic.

In this paper we prove the Alexandrov conjecture above in the special but important case of prescribed mean curvature spheres, i.e. when $\Phi=k_1+k_2$.  Observe that in this situation the immersed compact surface $\Sigma$ is not assumed to be strictly convex. In order to state our results, we fix some notation. From now on, all surfaces will be assumed to be (at least) of class $C^3$.

\begin{definition}\label{pmcr3}
\emph{Given $\cH\in C^1(\S^2)$, an immersed oriented surface $\Sigma$ in $\R^3$ with Gauss map
$\nu:\Sigma\flecha \S^2$ is said to have \emph{prescribed mean
curvature} $\cH$ if} $$H_{\Sigma} (p)= \cH(\nu(p))$$
\emph{for every $p\in \Sigma$, where $H_{\Sigma}$ is the mean
curvature of $\Sigma$.}
\end{definition}

\begin{theorem}\label{mainr3}
Let $S\subset \R^3$ be a strictly convex sphere with prescribed mean curvature
$\cH\in C^1(\S^2)$.
Then $S$ is (up to translations) the only immersed compact surface of genus zero in $\R^3$ with prescribed mean curvature $\cH$.
\end{theorem}

We will actually prove Theorem \ref{mainr3} as a particular case of a more general uniqueness theorem for immersed spheres of prescribed mean curvature in simply connected Riemannian homogeneous three-manifolds not isometric to the product space $\S^2(\kappa)\times \R$ (see Theorem \ref{main}). In particular, Theorem \ref{main} covers the situation of prescribed mean curvature spheres in $\H^3$ or $\S^3$, where no similar result seems to be known.

\begin{remark}
A classical result by Bonnet states that if there exists a diffeomorphism $\Psi:S_1\flecha S_2$ between two compact immersed surfaces $S_1,S_2$ of genus zero in $\R^3$ such that $\Psi$ preserves both the metric and the mean curvature function of the surfaces, then $S_1$ and $S_2$ are congruent in $\R^3$. Note that in Theorem \ref{mainr3} we are not assuming that the spheres $S$ and $\Sigma$ are isometric, which is a key hypothesis of Bonnet's problem.
\end{remark}

A famous theorem by H. Hopf (see for instance \cite{Ho}) asserts that any compact constant mean curvature surface of genus zero immersed in $\R^3$ is a round sphere. As a consequence of Theorem \ref{mainr3} we obtain a generalization of Hopf's theorem to the case of prescribed (not necessarily constant) antipodally symmetric mean curvature, as we explain next.

In their seminal paper \cite{GG}, B. Guan and P. Guan proved that if $\cH\in C^2(\S^2)$ satisfies
$\cH(-x)=\cH(x)>0$ for all $x\in \S^2$, then there exists a closed
strictly convex sphere $S_{\cH}$ in $\R^3$ with prescribed mean
curvature $\cH$. Note that when $\cH$ is constant, $S_{\cH}$ is a round sphere of radius $1/\cH$. Thus, the next corollary is a wide generalization of Hopf's theorem to the case of non-constant mean curvature:

\begin{corollary}\label{guanuni}
Let $\Sigma$ be an immersed compact surface of genus zero in $\R^3$ with prescribed mean
curvature $\cH\in C^2(\S^2)$, where $\cH(-x)=\cH(x)>0$. Then $\Sigma$ is the Guan-Guan sphere $S_{\cH}$ (up to translation).
\end{corollary}

Corollary \ref{guanuni} follows directly from Theorem \ref{mainr3} and the existence of the Guan-Guan strictly convex spheres in \cite{GG} mentioned above.

We have organized the paper as follows. In Section \ref{sec2} we
give some basic preliminaries about the geometry of simply connected
homogeneous three-manifolds $X$ not isometric to $\S^2(\kappa)\times \R$, and explain how they admit an underlying Lie group structure. That is, they can be seen as \emph{metric Lie groups}.
In Section \ref{sec3} we consider
conformally immersed surfaces $\psi:\Sigma\flecha X$ in metric Lie groups, and deduce an equation that links the Gauss map and
the mean curvature of $\psi$. This can be seen as a Weierstrass type
representation for surfaces in metric Lie groups with given mean
curvature and Gauss map, in the spirit of the classical Kenmotsu
formula \cite{Ke}.

In Section \ref{sec4} we prove our main uniqueness result (Theorem
\ref{main}) about prescribed mean curvature spheres in metric Lie groups $X$. Specifically, we prove that the existence of a compact surface $S$ of genus zero in $X$ with prescribed mean curvature $\cH\in C^1(\S^2)$ and whose Gauss map is a diffeomorphism to $\S^2$ implies the existence of a (non-holomorphic) complex quadratic differential for surfaces of prescribed mean curvature $\cH$ in $X$, which vanishes identically 	on open pieces of $S$ and only has isolated zeros of negative index for any
other surface. Theorem \ref{main} follows then from the existence of this
Hopf type differential by the Poincaré-Hopf theorem, and reduces to Theorem \ref{mainr3} when the homogeneous manifold $X$ is the Euclidean space $\R^3$.

In Section \ref{sec5} we give some final remarks on our results; in particular we consider the case where the genus of the compact surface $\Sigma$ is positive, and we show the necessity of the hypothesis that the Gauss map of $S$ is a diffeomorphism for Theorem \ref{main} and Theorem \ref{mainr3} to hold.

In the particular case that $\cH$ is constant, Theorem \ref{main}
follows from results by Hopf \cite{Ho} when $X$
has constant curvature, by Abresch and Rosenberg \cite{AR1,AR2} when
$X$ is rotationally symmetric, by Daniel and Mira \cite{DM} when
$X={\rm Sol}_3$, and by Meeks, Mira, Pérez and Ros \cite{mmpr0} for
general $X$. In particular, our proof is inspired by the study in \cite{DM,mmpr0} of that constant mean curvature case.

\section{Homogeneous three-manifolds}\label{sec2}

In this section we explain some basic geometric facts regarding
homogeneous three-manifolds. More specific details may be consulted
in ~\cite{mmpr0,mpe11}.

Let $\bar{M}^3$ be a homogeneous, simply connected Riemannian
three-manifold, and assume that $\bar{M}^3$ is not isometric to the
product space $\S^2(\kappa)\times \R$ of a two-dimensional sphere
$\S^2(\kappa)$ with the real line. Then $\bar{M}^3$ is diffeomorphic
to $\R^3$ or $\S^3$, and is isometric to a \emph{metric Lie group},
i.e. a three-dimensional simply connected Lie group $X$ furnished with a left
invariant metric $\langle ,\rangle $.

The isometry group of $\bar{M}^3$ has dimension six, four or three.
When the dimension is six, $\bar{M}^3$ has constant curvature. When
the dimension is four, $\bar{M}^3$ is rotationally symmetric, and is
one of the Riemannian fibrations $\Ek$, i.e. the product spaces
$\H^2(\kappa)\times \R$ and $\S^2(\kappa)\times \R$ for $\kappa\neq
0, \tau=0$, the Heisenberg space ${\rm Nil}_3$ for $\kappa=0,
\tau\neq 0$, and some rotational metrics on ${\rm SU}(2)$ or the
universal cover of ${\rm SL}(2,\R)$ if $\tau\neq 0$ and $\kappa\neq
0$. See \cite{D1} for the details. A generic homogeneous
three-manifold has isometry group of dimension three, and the
identity component is generated by the group of left translations of
$X$, when we view $X$ as a metric Lie group.

It is important to observe that the homogeneous three-manifolds
$\R^3, \H^3$ and $\E(\kappa,\tau)$ with $\kappa<0$ admit more than
one Lie group structure for which the metric is left invariant. That
is, these homogeneous three-manifolds are isometric to at least two
metric Lie groups $X,X'$ that are non-isomorphic as Lie groups.

For computational purposes, it will be useful to divide the class of
metric Lie groups $X$ into two cases: \emph{unimodular} metric Lie
groups and \emph{non-unimodular} metric Lie groups. (See
\cite{mpe11} for some equivalent definitions of unimodularity,
although we will not use the concept itself, just the resulting
classification of Lie groups). We note that $\R^3$ and $\S^3$ are
unimodular, while $\H^3$ is non-unimodular.

\subsection{Unimodular metric Lie groups.}
\label{unimod} Let $X$ be a three-dimensional unimodular metric Lie
group. Then, there exists a left invariant orthonormal frame
$\{E_1,E_2,E_3\}$ in $X$ which satisfies the following structure
equations:
\begin{equation}
\label{eq:11}
 [E_2,E_3]=c_1E_1,\quad [E_3,E_1]=c_2E_2, \quad
[E_1,E_2]=c_3E_3,
\end{equation}
for certain constants $c_1,c_2,c_3\in \R$, among which  at most one
$c_i$  is negative. We call $\{E_1,E_2,E_3\}$ the \emph{canonical
frame} of $X$.

Two unimodular metric Lie groups with the same structure constants
are isometric and isomorphic. Two unimodular metric Lie groups with
the same signature for the triple $(c_1,c_2,c_3)$ are isomorphic,
but not isometric in general. If $c_1=c_2=c_3$, then $X$ has
constant curvature. If two of the constants $c_i$ coincide, $X$ has
an isometry group of dimension four, and hence is rotationally
symmetric.

The table below shows the six possible different Lie group
structures depending on the signature of $(c_1,c_2,c_3)$. Each
horizontal line corresponds to a unique Lie group structure; when
all the structure constants are different, the isometry group of $X$
is three-dimensional.

 \par
\vspace{.1cm}
 \begin{center}

 \begin{tabular}{|c|c|c|c|}
        \hline
        Signs of $c_1,c_2,c_3$ & $\dim \mbox{Isom}(X)=3$ & $\dim
        \mbox{Isom}(X)=4$
        & $\dim \mbox{Isom}(X)=6$ \\ \hline
\rule{0cm}{.4cm}
        +,\ +,\ + & $\su$ & $\esf^3_{\mbox{\tiny Berger}}
        =\E (\kappa,\tau ), \kappa>0 $ & $\esf^3(\kappa )$ \\
\rule{0cm}{.4cm}
        +,\ +,\ -- & $\sl $ & $\E (\kappa,\tau ), \kappa<0$ & $\emptyset$\\
\rule{0cm}{.4cm}
        +,\ +,\ 0 & $\widetilde{\mbox{\rm E}}(2)$ & $\emptyset$
        & $(\widetilde{\mbox{\rm E}}(2),\mbox{flat})$ \\
\rule{0cm}{.4cm}
        +,\ --,\ 0 & Sol$_3$ & $\emptyset$ & $\emptyset$\\
\rule{0cm}{.4cm}
        +,\ 0,\ 0 & $\emptyset$ & Nil$_3=\E (0,\tau )$, $\tau \neq 0$ & $\emptyset$  \\
\rule{0cm}{.4cm}
        0,\ 0,\ 0 & $\emptyset$ & $\emptyset$ & $\R^3$ \\
        \hline
      \end{tabular}
      \\ \mbox{}
\end{center}
{\sc Table 1.} All three-dimensional, simply connected unimodular metric
Lie groups. Here, $\sl$ is the universal cover of $\mathrm{SL}(2,\R)
$, $\EE $ the universal cover of the group of orientation preserving
Euclidean isometries of $\R^2$, $\sol$ is the universal cover of the
group of orientation preserving isometries of the Lorentzian plane,
Nil$_3$ is Heisenberg group of real upper triangular $3\times3$
matrices, and $\R^3$ is the abelian group.
\par
\vspace{.2cm}

If we write

\begin{equation} \label{def:mu}
\mu _1=\frac{1}{2}(-c_1+c_2+ c_3),\quad  \mu _2=\frac{1}{2}(c_1-c_2+
c_3),\quad  \mu _3=\frac{1}{2}(c_1+c_2-c_3),
\end{equation}
the Riemannian connection of $(X,\esiz,\esde)$ is given by
\begin{equation}
\label{eq:LCunim} \nabla _{E_i}E_j=\mu _i\, E_i\times E_j,\quad
i,j\in \{ 1,2,3\}.
\end{equation}
Here, $\times $ denotes the cross product associated to $\langle
,\rangle $ and to the orientation on $X$ defined by declaring
$(E_1,E_2,E_3)$ to be a positively oriented basis.

\subsection{Non-unimodular metric Lie groups.}\label{nonunims}

Let $X$ be a three-dimensional, simply connected non-unimodular
metric Lie group. Then we can view $X$ as a semi-direct product
$X=\R^2\rtimes_A \R$ endowed with its canonical metric, as we
explain next.

Let $A$ be a $2\times 2$ matrix with trace $2$, which we write in
the form \begin{equation} \label{Axieta} A=A(a,b)= \left(
\begin{array}{cc}
1+a & -(1-a)b\\
(1+a)b & 1-a \end{array}\right), \hspace{1cm} a,b\in [0,\infty).
\end{equation}
Then we can consider the metric Lie group $\R^2\rtimes_A
\R$ given as $(\R^3\equiv \R^2\times \R,*,\esiz,\esde)$ where:
 \begin{enumerate}
 \item
The Lie group operation $*$ is given by
\begin{equation}
\label{eq:5}
 ({\bf p}_1,z_1)*({\bf p}_2,z_2)=({\bf p}_1+ e^{z_1 A}\  {\bf
 p}_2,z_1+z_2).
\end{equation}
 \item
The \emph{canonical metric} $\esiz,\esde$ is the left invariant
metric on $\R^2\rtimes_A \R$ (for the product $*$ above) defined by
extending the usual inner product of $\R^3$ at the origin to the
whole space through the left invariant frame $\{ E_1,E_2,E_3\} $ of
$\R^2\rtimes_A \R$ given by
\begin{equation}
\label{eq:6*}
 E_1(x,y,z)=a_{11}(z)\partial _x+a_{21}(z)\partial _y,\quad
E_2(x,y,z)=a_{12}(z)\partial _x+a_{22}(z)\partial _y,\quad
 E_3=\partial _z,
\end{equation}
where
\begin{equation}
\label{eq:exp(zA)}
 e^{zA}=\left(
\begin{array}{cr}
a_{11}(z) & a_{12}(z) \\
a_{21}(z) & a_{22}(z)
\end{array}\right).
\end{equation}
 \end{enumerate}
In this way, $\{E_1,E_2,E_3\}$ becomes a left invariant orthonormal
frame on $\R^2\rtimes_A \R$, which we call the \emph{canonical
frame} of the space.

In terms of $A$, the Lie bracket relations are:
\begin{equation}
\label{eq:8a} [E_1,E_2]=0, \quad [E_3,E_1]=(1+a)E_1+b (1+a) E_2,
\quad
 [E_3,E_2]=b(a-1) E_1+(1-a) E_2.
 \end{equation}
From there, the Levi-Civita connection of $\R^2\rtimes_A \R$ is
given by {\large
\begin{equation}
\label{eq:12}
\begin{array}{l|l|l}
\rule{0cm}{.5cm} \nabla _{E_1}E_1=(1+a)\, E_3 & \nabla
_{E_1}E_2=ab\, E_3
& \nabla _{E_1}E_3=-(1+a)\, E_1-ab \, E_2 \\
\rule{0cm}{.5cm} \nabla _{E_2}E_1=ab\, E_3 & \nabla
_{E_2}E_2=(1-a)\, E_3
& \nabla _{E_2}E_3=-ab\, E_1-(1-a)\, E_2 \\
\rule{0cm}{.5cm} \nabla _{E_3}E_1=b\, E_2 & \nabla _{E_3}E_2=-b\,
E_1 & \nabla _{E_3}E_3=0.
\end{array}
\end{equation}
}

The \emph{Cheeger constant} ${\rm Ch}(\R^2\rtimes_A \R)$ of
$\R^2\rtimes_A \R$ is ${\rm trace}(A)=2$. We also note that every
leaf of the foliation $\mathcal{F}= \{ \R^2\rtimes _A\{ z\} \mid
z\in \R \}$ has constant mean curvature $H=\mbox{trace}(A)/2=1$ with
respect to the unit normal vector field $E_3$. In particular, by the
mean curvature comparison principle, there are no immersed compact
surfaces in $\R^2\rtimes_A \R$ with mean curvature function $|H|>1$
at every point. In other words, the \emph{critical mean curvature}
of $\R^2\rtimes_A \R$ is $1$ (see \cite{mpe11,mmpr0,mmpr2}).

This construction of $\R^2\rtimes_A \R$ that we have just carried out
recovers (up to homothety) all non-unimodular metric Lie groups:

{\bf Fact (see \cite{mpe11}):} Let $X$ be a simply connected
non-unimodular three-dimensional metric Lie group, with its metric
rescaled so that ${\rm Ch} (X)=2$. Then $X$ is isomorphic and
isometric to the semi-direct $\R^2\rtimes_A \R$ with its canonical
metric and $A$ given by \eqref{Axieta} for some $a,b\in [0,\8)$,

The hyperbolic three-space $\H^3$ of constant curvature $-1$ is the
semi-direct product $\R^2\rtimes_A \R$ with $A=I_2$. Similarly, if
$a=1,b=0$ we recover the product space $\H^2(-4)\times \R$.

\subsection{The Gauss map for surfaces in metric Lie groups}

Let $\psi:\Sigma\flecha X$ be an immersed oriented surface in a metric Lie group $X$, and let $N:\Sigma\flecha TX$ denote its unit normal. Note that for any $x\in X$ the left translation $l_x :X\flecha X$ is an isometry of $X$. Thus, for every $p\in \Sigma$ there exists a unique unit vector $\nu(p)\in T_e X$, $|\nu(p)|=1$, such that $$ dl_{\psi(p)} (\nu(p)) = N(p), \hspace{1cm }\forall p\in \Sigma,$$ where $e$ denotes the identity element of $X$.

\begin{definition}\label{defgm}
We call the map $\nu:\Sigma\flecha \S^2=\{v\in T_e X : |v|=1\}$ the \emph{Gauss map} of the oriented surface $\psi:\Sigma\flecha X$.
\end{definition}

The Gauss map $\nu$ can be easily written in coordinates as follows: let $\{E_1,E_2,E_3\}$ be a left invariant orthonormal frame of $X$, and write $N=\sum_{i=1}^3 \nu_i E_i$ for the unit normal $N$ of $\psi$. Then the Gauss map $\nu:\Sigma\flecha \S^2\subset T_e X$ is written with respect to the orthonormal basis of the Lie algebra $\{(E_1)_e,(E_2)_e,(E_3)_e\}$ as $\nu= (\nu_1,\nu_2,\nu_3):\Sigma\flecha \S^2$.

Note that if $X$ is the Euclidean three-space $\R^3$, we have $dl_x ={\rm Id}$ for every $x\in \R^3$; thus, the Gauss map $\nu$ in Definition \ref{defgm} is the natural extension to metric Lie groups of the usual Gauss map of surfaces in $\R^3$.

Also note that we can extend the notion of surfaces with prescribed mean curvature in terms of the Gauss map in $\R^3$ (see Definition \ref{pmcr3}) to the case of metric Lie groups:

\begin{definition}\label{pmcx}
\emph{Let $X$ be a metric Lie group, and $\cH\in C^1(\S^2)$. An immersed oriented surface $\Sigma$ in $X$ with Gauss map
$\nu:\Sigma\flecha \S^2$ is said to have \emph{prescribed mean
curvature} $\cH$ if} $$H_{\Sigma} (p)= \cH(\nu(p))$$
\emph{for every $p\in \Sigma$, where $H_{\Sigma}$ is the mean
curvature of $\Sigma$.}
\end{definition}

\begin{remark}
When $X$ is the hyperbolic three-space $\H^3$ or the sphere $\S^3$, the left invariant Gauss map $\nu:\Sigma\flecha \S^2$ of Definition \ref{pmcx} is usually called the \emph{normal Gauss map}. The problem of prescribing a mean curvature function and the normal Gauss map in $\H^3$ or $\S^3$ has been treated, for instance, in \cite{Ko,AA1,AA2}.

A different but also natural choice of Gauss map for surfaces in $\H^3$ (which we do not treat here) is the \emph{hyperbolic Gauss map}; see \cite{EGM} for the study of a Christoffel-Minkowski problem in $\H^{n+1}$ in terms of the hyperbolic Gauss map.
\end{remark}

\begin{remark}
In $\R^3$ the Gauss map of a surface $S\subset \R^3$ is a diffeomorphism into $\S^2$ if and only if $S$ is a strictly convex ovaloid. In particular, $S$ is embedded. This is not true in general when we substitute $\R^3$ by a metric Lie group $X$. For instance, in some homogeneous three-manifolds diffeomorphic (but not isometric) to $\S^3$ there exist constant mean curvature spheres which are not embedded, but whose Gauss maps are diffeomorphisms into $\S^2$ (see \cite{mmpr0,To}).
\end{remark}

\subsection{The potential function of $X$}

The next definition is a slight reformulation of the concept of
$H$-potential of a three-dimensional metric Lie group in
\cite{mmpr0}. It will play an important role in our computations in
the next two sections.

\begin{definition}
\label{HP} The \emph{potential function} of the space $X$ is the map
$R(H,q):\R\times\bar{\C}\flecha \bar{\C}$ given by
 \begin{equation}\label{defpot}
R(H,q)= H(1+|q|^2)^2 + \Theta (q),
 \end{equation}
where:
 \begin{enumerate}
\item If $X$ is unimodular, then $$\Theta (q) =
-\frac{i}{2}\left(
 \mu_2|1+q^2|^2+\mu _1|1-q^2|^2+4\mu _3|q|^2\right),$$ where $\mu _1,\mu _2,\mu _3\in \R$ are the
related numbers defined in~(\ref{def:mu}).
  \item
If $X$ is non-unimodular, and we rescale its metric as explained in
Subsection \ref{nonunims} so that ${\rm Ch} (X) =2$, then $$ \Theta
(q)= -(1-|q|^4)-a\left( q^2-\overline{q}^2\right) -ib \left(
2|q|^2-a\left( q^2+\overline{q}^2\right) \right)$$ where $a,b \geq
0$ are the constants appearing in~(\ref{Axieta}) when we view $X$ as
the semi-direct product $\R^2 \rtimes_A \R$.
 \end{enumerate}
\end{definition}

We will say that the potential $R$ for $X$ has a zero at $q_0=\8\in
\overline{\C}$ if $\lim_{q\to\8} R(H,q)/|q|^4=0$ for every $H$.

The zeros of $R$ are related to the existence of two-dimensional
subgroups of $X$, as follows: $R(H_0,q_0)=0$ in $X$ if and only if
there exists a two dimensional subgroup in $X$ with (constant) Gauss
map $q_0$ and (constant) mean curvature $H_0$ (see Corollary 3.17 in \cite{mpe11}).

The potential function $R$ has no zeros if $X$ is compact (i.e. if
$X$ is diffeomorphic to $\S^3$). If $X$ is unimodular and $H_0\neq
0$, then $R(H_0,q)\neq 0$ for all $q\in \bar{\C}$. If $X$ is
non-unimodular, rescaled so that ${\rm Ch}(X)=2$ as explained in
Subsection \ref{nonunims}, and if $|H_0|>1$, then $R(H_0,q)\neq 0$
for every $q\in \bar{\C}$. Thus we have:

\begin{lemma}\label{regR}
Assume that $X$ is not compact, and let $\mathfrak{h}_0(X)\geq 0$ be
the number given by $$\left\{ \begin{array}{lll} \text{
$\mathfrak{h}_0(X)=0$} & \text{if} & \text{$X$ is unimodular.} \\
\text{ $\mathfrak{h}_0(X)=1$} & \text{if} & \text{$X$ is
non-unimodular, rescaled to ${\rm Ch}(X)=2$.} \end{array}\right.$$
Then $R(H,q)\neq 0$ for every $q\in \bar{\C}$ and every $H$ with
$|H|>\mathfrak{h}_0(X).$
\end{lemma}

\section{An elliptic PDE for the Gauss map}\label{sec3}

A theorem by Kenmotsu \cite{Ke} proves that a necessary and
sufficient condition for a map $g:\Sigma\flecha \bar{\C}$ from a
simply connected Riemann surface $\Sigma$ to be the Gauss map of a
conformal immersion $\psi:\Sigma\flecha \R^3$ with a given mean
curvature function $H:\Sigma\flecha (0,\8)$ is that $$g_{z\bar{z}}=
\frac{2 \bar{g}}{1+|g|^2} g_z g_{\bar{z}} + \frac{H_{\bar{z}}}{H}
g_z,$$ for any conformal parameter $z$ of $\Sigma$. Moreover, if this equation holds, the immersion $\psi$ can be
recovered from $g,H$ by an integral representation formula.

In this section we extend this theorem to the general case where the
ambient space is an arbitrary simply connected homogeneous
three-manifold $X$ not isometric to $\S^2(\kappa)\times \R$. This
also extends the Weierstrass type representation for CMC surfaces in
homogeneous manifolds in \cite{mmpr0}, and the work by Aiyama and
Akutagawa \cite{AA1,AA2} for prescribed non-constant mean curvature
surfaces in $\H^3$ and $\S^3$.

In the next result $\Sigma$ denotes a Riemann surface and $z$ an
arbitrary conformal parameter on $\Sigma$. We identify the Gauss map
$\nu:\Sigma\flecha \S^2$ of an oriented surface $\psi:\Sigma\flecha
X$ with its south pole stereographic projection $g$, i.e. if $N =
\sum \nu_i E_i$ with respect to the canonical frame
$\{E_1,E_2,E_3\}$ of $X$, then
\begin{equation}\label{genu}
g= \frac{\nu_1 + i \nu_2}{1+\nu_3}:\Sigma\flecha \bar{\C}.
\end{equation}
Also, $R(H,q):\R \times \bar{\C}\flecha \bar{\C}$ will denote the
potential function of $X$ (see Definition \ref{HP}).
\begin{theorem}\label{rep}
Let $\psi:\Sigma\flecha X$ be a conformally immersed oriented
surface, and let $H:\Sigma\flecha \R$ and $g:\Sigma\flecha \bar{\C}$
denote its mean curvature and Gauss map, respectively. Then $(g,H)$
satisfy the conformally invariant complex elliptic PDE

\begin{equation}\label{gaussmap}
g_{z\bar{z}} = \frac{R_q}{R}(H,g) g_z g_{\bar{z}} + \left(
\frac{R_{\bar{q}}}{R} -
\frac{\overline{R_q}}{\overline{R}}\right)(H,g) |g_z|^2 +
\frac{R_H}{R} (H,g) H_{\bar{z}} g_z,
\end{equation}
at all points $p\in \Sigma$ with $R(H(p),g(p))\neq 0$.

Conversely, let $\Sigma$ be simply connected, and let
$H:\Sigma\flecha \R$ and $g:\Sigma \flecha \bar{\C}$ satisfy:

\begin{enumerate}
  \item
$R(H(p),g(p))\neq 0$ for every $p\in \Sigma$.
 \item
$g_z(p)\neq 0$ for every\footnote{If $g(p)=\8$, the condition $g_z(p)\neq 0$ should be interpreted as $\lim_{z\to p} g_z(p)/g(p)^2 \neq 0$.} $p\in \Sigma$.
 \item
$(H,g)$ are a solution to \eqref{gaussmap}.
\end{enumerate}
Then there exists a conformal immersion $\psi:\Sigma\flecha X$,
unique up to left translations, with Gauss map $g$ and mean
curvature $H$.
\end{theorem}
\begin{proof}
We will only prove the result for the case that $X$ is unimodular; the argument when $X$ is non-unimodular is exactly the same but some intermediate expressions are different because of the difference between the potential functions of unimodular and non-unimodular spaces.

Let $\psi:\Sigma \flecha X$ be a conformal immersion with unit normal $N:\Sigma\flecha TX$, and denote
$\esiz d\psi,d\psi\esde = \landa |dz|^2$, where $z$ is a local
conformal parameter on $\Sigma$. Let $\{E_1,E_2,E_3\}$ be the
canonical left invariant orthonormal frame of $X$, as explained in
Section \ref{sec2}, $\nu:\Sigma\flecha \S^2$ be the Gauss map of
$\psi$ and $g$ be the Gauss map after stereographic projection
\eqref{genu}. We will work around a point $p\in \Sigma$ where $g\neq
0,\8$ and $R(H,g)\neq 0$. We can write
 \begin{equation}\label{asubi}
\psi_z =\sum_{i=1}^3 A_i
(E_i \circ \psi), \hspace{ 0.5cm} \psi_{\bar{z}} =\sum_{i=1}^3
\overline{A_i} (E_i \circ \psi), \hspace{0.5cm} N =\sum_{i=1}^3
\nu_i (E_i \circ \psi),
\end{equation}
for smooth functions $A_i:\Sigma\flecha \C$,
$\nu_i:\Sigma\flecha \R$, $i=1,2,3$. Noting that, from \eqref{genu},
\begin{equation}\label{genu2}
(\nu_1,\nu_2,\nu_3)=\frac{1}{1+|g|^2} (g+\bar{g},i(\bar{g}-g),1-|g|^2),
 \end{equation}
it follows easily from the metric relations $\esiz \psi_z,\nu\esde =
\esiz \psi_{\bar{z}},\nu\esde =0$, $\esiz \nu,\nu\esde=1$ that, if
we denote $A_3= \eta/2$, then
\begin{equation}\label{rf1}
A_1= \frac{\eta}{4}\left(\overline{g} - \frac{1}{\overline{g}}\right), \hspace{0.5cm} A_2 = \frac{i \,\eta}{4}\left(\overline{g} + \frac{1}{\overline{g}}\right), \hspace{0.5cm} A_3 = \frac{\eta}{2}.
\end{equation}
From here,
 \begin{equation}\label{rf2}
 \landa= 2\sum_{i=1}^3 |A_i|^2 = \frac{(1+|g|^2)^2 |\eta|^2}{4|g|^2}.
 \end{equation}

If we let $\nabla$ denote the Levi-Civita connection of $X$, a classical computation from the Gauss-Weingarten equations gives
 \begin{equation}\label{napi}
\nabla_{\psi_{\bar{z}}} \psi_z = \frac{\landa H}{2} \nu.
 \end{equation}
If we now express $\psi_z$ and $\psi_{\bar{z}}$ as in \eqref{asubi} and use the relations between $\nabla$ and $\{E_1,E_2,E_3\}$ given in \eqref{eq:LCunim}, we can write \eqref{napi} in coordinates with respect to $\{E_1,E_2,E_3\}$. Let us use brackets to denote coordinates in the $\{E_1,E_2,E_3\}$ basis. Then, we obtain
 \begin{equation}\label{rf3} \def\arraystretch{1.6} \begin{array}{lll} [(A_1)_{\bar{z}}, (A_2)_{\bar{z}},(A_3)_{\bar{z}}] & = & -\sum_{i,j} \overline{A_i} A_j \nabla_{E_i} E_j + \frac{\landa H}{2} [\nu_1,\nu_2,\nu_3]  \\ & = & [\mu_3 A_2 \overline{A_3} - \mu_2 A_3 \overline{A_2},\mu_3 A_1 \overline{A_3}-\mu_1 A_3 \overline{A_1}, \mu_2 A_1 \overline{A_2}-\mu_1 A_2 \overline{A_1} ]  \\ & & + \frac{\landa H}{2} [\nu_1,\nu_2,\nu_3]. \end{array}\end{equation}
The third equation in \eqref{rf3} gives, using \eqref{rf1} and
\eqref{genu2},
 \begin{equation}\label{rf4}
 \frac{4\eta_{\bar{z}}}{|\eta|^2} = \frac{H (1-|g|^4)}{|g|^2} - \frac{i}{2} \left\{ \mu_1 \left(\frac{1}{\overline{g}} + \overline{g}\right)\left(-\frac{1}{g} + g\right) - \mu_2 \left(\frac{1}{\overline{g}} - \overline{g}\right)\left(\frac{1}{g} + g\right)\right\}
 \end{equation}
A similar process can be done with $(1{\rm st})+i(2{\rm nd})$
equation in \eqref{rf3}, using again \eqref{rf1} and \eqref{genu2}.
In this way we obtain
 \begin{equation}\label{rf40}
 \frac{4\eta_{\bar{z}}}{|\eta|^2} = \frac{4 \overline{g_z}}{\bar{g}\bar{\eta}} - 2H(1+|g|^2)-i \left\{ \mu_1 \left(|g|^2-\frac{\overline{g}}{g}\right) +  \mu_2 \left(|g|^2 + \frac{\overline{g}}{g} \right) +2 \mu_3\right\}.
 \end{equation}
By comparing \eqref{rf40} with
\eqref{rf4}, a computation provides the following expression for $\eta$:
 \begin{equation}\label{rf10}
 \eta= \frac{4 \overline{g} g_z}{R(H,g)}.
 \end{equation}
Differentiating \eqref{rf10}, we get
 \begin{equation}\label{rf11}
 \frac{\eta_{\bar{z}}}{\eta} = \frac{\overline{g_z}}{\overline{g}} +\frac{g_{z\bar{z}}}{g_z} - \frac{R_q}{R}(H,g) g_{\bar{z}} -  \frac{R_{\bar{q}}}{R}(H,g) \overline{g_z} -\frac{R_H}{R} (H,g) H_{\bar{z}}.
 \end{equation}
Finally, comparing \eqref{rf11} with \eqref{rf4} and using (the
conjugate of) \eqref{rf10}, we obtain the Gauss map equation
\eqref{gaussmap}. By smoothness, it also holds when $g=0,\8$ (that
\eqref{gaussmap} extends to the points where $g=\8$ can be easily
checked by working at those points with the smooth function
$\varphi= 1/g$). This completes the proof of the first statement of
the theorem.

The proof of the converse statement follows by a computation using
the Frobenius theorem. Specifically, the argument follows very
closely the proof of Theorem 3.7 in \cite{mmpr0}, which covers the
case where $H$ is constant. We give an outline next.

Let $g:\Sigma\flecha \bar{\C}$ and $H:\Sigma\flecha \R$ satisfy conditions (1), (2), (3) as in the statement of Theorem \ref{rep}, and assume that $\Sigma$ is simply connected. We define $A_1,A_2,A_3:\Sigma\flecha \C$ as in \eqref{rf1}, where $\eta$ is given by \eqref{rf10}; that all the $A_i$'s have finite value at every $p\in \Sigma$ follows from the fact that $R(H,q)/|q|^4$ has a finite limit at $q=\8$. We compute

\begin{equation}\label{eaes}\def\arraystretch{2} \begin{array}{lll} (A_1)_{\bar{z}} & = & \displaystyle \frac{\bar{g}^2 -1}{R(H,g)} g_{z\bar{z}} - \frac{\bar{g}^2-1}{R(H,g)^2} \left( R(H,g)\right)_{\bar{z}} g_z  + \displaystyle \frac{2\bar{g}}{R(H,g)} |g_z|^2 \\  & = &\displaystyle \frac{\bar{g}^2 -1}{R(H,g)} g_{z\bar{z}} - \frac{\bar{g}^2-1}{R(H,g)^2} \left( R_q (H,g) g_{\bar{z}} + R_{\bar{q}} (H,g) \overline{g_z} + R_H (H,g) H_{\bar{z}} \right) g_z \\ & & + \displaystyle \frac{2\bar{g}}{R(H,g)} |g_z|^2.\end{array}\end{equation} Using in this expression that $g$ is a solution to \eqref{gaussmap}, we get
 \begin{equation}\label{ecig}
 (A_1)_{\bar{z}} = \frac{|g_z|^2}{|R(H,g)|^2} \left(\overline{2 g R(H,g) - R_q (H,g) (g^2-1)}\right).
 \end{equation}
Observe that the derivative $H_{\bar{z}}$ in \eqref{eaes} cancels in \eqref{ecig} with the $H_{\bar{z}}$ appearing after substituting $g_{z\bar{z}}$ by its value in \eqref{gaussmap}. The same happens if we work with $A_2$ or $A_3$ instead of $A_1$.

Note that \eqref{ecig} is exactly the same formula as (3.5) in \cite{mmpr0} (with the change of notation of substituting $R(g)$ there by $R(H,g)$ here). In other words, formula (3.5) in \cite{mmpr0} also holds when $H$ is not constant.

At this point, the rest of the proof of the direct statement in Theorem 3.7 of \cite{mmpr0} never uses again that $H$ is constant. Thus, it translates essentially word by word to our situation, up to the change of notation $R(g) \leftrightarrow R(H,g)$ explained above. This finishes the proof of Theorem \ref{rep}.
\end{proof}

\subsection{Remarks}\label{subrem}

\begin{enumerate}
\item
In the proof of Theorem \ref{rep} we established that the metric of
a conformal immersion $\psi:\Sigma\flecha X$ with mean curvature
$H:\Sigma\flecha \R$ and Gauss map $g$ is
 \begin{equation}\label{eqlanda}
\esiz d\psi, d\psi\esde = \landa |dz|^2, \hspace{0.5cm} \landa =
\frac{4 (1+|g|^2)^2}{|R(H,g)|^2} |g_z|^2. \end{equation} Thus, if
$R(H(p),g(p))\neq 0$ for some $p\in\Sigma$, then $g_z(p)\neq 0$.
 \item
In particular, it follows from Lemma \ref{regR} that if $X$ is
compact, or if the mean curvature of $\psi:\Sigma\flecha X$
satisfies $|H|>\mathfrak{h}_0(X)$ when $X$ is non-compact, then
$g_z\neq 0$ everywhere on $\Sigma$.
 \item
The converse of Theorem \ref{rep} provides a Weierstrass
representation formula which recovers a conformal immersion
$\psi:\Sigma\flecha X$ (with $|H|>\mathfrak{h}_0(X)$ if $X$ is
non-compact) from its mean curvature function $H$ and its Gauss map
$g$. Specifically, if $\Sigma$ is simply connected and
$(g,H):\Sigma\flecha \bar{\C}\times \R$ satisfy conditions (1), (2), (3) in the statement of Theorem \ref{rep}, then
we can define $A_i:\Sigma\flecha \C$, $i=1,2,3$, by $$A_1= \frac{\eta}{4}\left(\overline{g} - \frac{1}{\overline{g}}\right),
\hspace{0.5cm} A_2 = \frac{i \,\eta}{4}\left(\overline{g} + \frac{1}{\overline{g}}\right), \hspace{0.5cm} A_3 = \frac{\eta}{2}
, \hspace{0.5cm} \eta= \frac{4 \overline{g} g_z}{R(g)}.
 $$ and the immersion
$\psi:\Sigma\flecha X$ with mean curvature $H$ and Gauss map $g$ may
be recovered from $(g,H)$ by integrating $$\psi_z = \sum_{i=1}^3 A_i
(E_i\circ \psi)$$ on $\Sigma$. Here, $\{E_1,E_2,E_3\}$ is the
canonical frame of $X$. The final formula for $\psi$ involves
intricate integral expressions in terms of the coordinates of $\psi$
and the structure constants of $X$, and so will be omitted here.
 \item
As a consequence of the previous remark we have the following \emph{uniqueness result}: Let $\psi_1,\psi_2:\Sigma\flecha X$ be two conformal immersions with the same mean curvature $H:\Sigma\flecha \R$, the same Gauss map $g:\Sigma\flecha \bar{\C}$, and so that $R(H(z),g(z))\neq 0$ for every $z\in \Sigma$. Then $\psi_2= L \circ \psi_1$ for some left translation $L:X\flecha X$.
 \item
The usual Hopf differential $P\, dz^2$ of $\psi:\Sigma\flecha X$ is defined as $$P= -\esiz \nabla_{\psi_z} N,\psi_z\esde,$$ and can be computed using \eqref{genu2}, \eqref{rf1}, \eqref{rf10} as follows:
\begin{equation}\label{calho}
\def\arraystretch{2} \begin{array}{lll}P & =&  - \displaystyle \sum_{i=1}^3 A_i (\nu_i)_z - \sum_{i,j,k=1}^3 A_i \nu_j A_k \esiz \nabla_{E_i} E_j,E_k\esde \\
 & = &\displaystyle \frac{2}{R(H,g)} g_z \bar{g}_z - \sum_{i,j,k=1}^3 \gamma_{ij}^k A_i \nu_j A_k, \end{array}
\end{equation}
where $\gamma_{ij}^k:= \esiz \nabla_{E_i} E_j,E_k\esde$ is a constant depending on the Lie algebra structure of $X$. Thus, using again the relations \eqref{genu2}, \eqref{rf1}, \eqref{rf10} we see that $P$ can be written in the form
\begin{equation}\label{calho2}
P= \frac{2}{R(H,g)} g_z \bar{g}_z + U(g,\bar{g}) g_z^2,
\end{equation}
where $U(g,\bar{g})$ is a rational expression in $g,\bar{g}$ whose coefficients depend on $\gamma_{ij}^k$.
\end{enumerate}

\section{Prescribed mean curvature spheres: proof of Theorem \ref{main}}\label{sec4}

Let $X$ be a three-dimensional metric Lie group, and $\cH\in
C^1(\S^2)$. In this section we prove:

\begin{theorem}\label{main}
Let $S$ be an immersed sphere in $X$ with prescribed mean curvature
$\cH\in C^1(\S^2)$, $\cH>0$, and assume that the Gauss map $\nu:S\flecha
\S^2$ is a diffeomorphism.

Then any other immersed sphere $\Sigma$ in $X$ with prescribed mean
curvature $\cH$ is a left translation of $S$. In particular,
$\Sigma$ and $S$ are congruent in $X$.
\end{theorem}

\begin{remark}
Theorem \ref{main} clearly implies Theorem \ref{mainr3}. Note that under the hypotheses of Theorem \ref{mainr3}, the prescribed function $\cH\in C^1(\S^2)$ needs to be positive at every point.
\end{remark}

To begin the proof of Theorem \ref{main}, assume that there exists an immersed sphere $S$ in $X$
with prescribed mean curvature $\cH$, and whose Gauss map
$\nu:S\flecha \S^2$ is a diffeomorphism. Let $G=\pi\circ\nu:S\flecha
\bar{\C}$ where $\pi$ denotes the stereographic projection from the
south pole, i.e. if the unit normal $N$ of $\psi$ is expressed as $N=\sum_i \nu_i E_i$
with respect to the canonical frame $\{E_1,E_2,E_3\}$ of $X$, then
$$G=\frac{\nu_1 +i \nu_2}{1+\nu_3}:S\flecha \bar{\C}.$$ Then $G$ is
an orientation preserving diffeomorphism, and so
$|G_z|^2-|G_{\bar{z}}|^2>0$. In particular $G_z\neq 0$ at every
point. By \eqref{eqlanda}, the potential function $R(H,q)$ of $X$ does
not vanish at the points of the form $(H_S(p),G(p))$, where $p\in S$
and $H_S$ is the mean curvature function of $S$. So, as $G$ is a
diffeomorphism, $R(\cH(q),q)\neq 0$ for every $q\in \bar{\C}$.

Theorem \ref{main} follows easily from the following result of
independent interest:

\begin{theorem}\label{hopf}
In the conditions of Theorem \ref{main}, there exists a complex quadratic
differential $Q_{\cH} \, dz^2$ defined for any immersed surface
$\psi:\Sigma\flecha X$ with prescribed mean curvature $\cH$, so
that:

\begin{enumerate}
\item
$Q_{\cH}\, dz^2$ vanishes identically on $S$.
 \item
If $Q_{\cH}\, dz^2$ vanishes identically on $\psi:\Sigma\flecha X$,
then $\psi(\Sigma)$ is a left translation in $X$ of an open subset
of $S$.
 \item
If $Q_{\cH}\, dz^2$ does not vanish identically on
$\psi:\Sigma\flecha X$, then the zeros of $Q_{\cH} \, dz^2$ are all
isolated and of negative index on $\Sigma$.
\end{enumerate}
\end{theorem}

Indeed, by the Poincaré-Hopf theorem, a complex quadratic differential on
$\bar{\C}$ cannot have only isolated zeros of negative index. Thus
Theorem \ref{main} follows from Theorem \ref{hopf}.

\vspace{0.2cm}

\begin{proof}[Proof of Theorem \ref{hopf}]

Define
 \begin{equation}\label{defM}
\cM(q)= \frac{1}{R(\cH(q),q)}: \bar{\C} \flecha \C,
 \end{equation}
 which by the discussion above takes finite values, and let $\cL(q):\bar{\C}\flecha \C$
be
 \begin{equation}\label{defL}
\cL(q)= -\frac{\bar{G}_z}{G_z} (G^{-1} (q)) \cM(q).
 \end{equation}

We define for any conformally immersed surface
$\psi:\Sigma\flecha X$ with prescribed mean curvature $\cH$ the
complex quadratic differential $Q_{\cH} \, dz^2$ on $\Sigma$ given
by
 \begin{equation}\label{hopfdif}
Q_{\cH} = \cL (g) g_z^2 + \cM(g) g_z \bar{g}_z.
 \end{equation}
Here $z$ is an arbitrary conformal parameter of $\Sigma$, and
$g:\Sigma\flecha \bar{\C}$ is the Gauss map of $\psi$.

We divide the proof of Theorem \ref{hopf} into several claims.

\vspace{0.2cm}

{\bf Claim 1:} \emph{$Q_{\cH}\, dz^2$ is a well defined complex
quadratic differential on any surface $\psi:\Sigma\flecha X$ with
prescribed mean curvature $\cH$.}
\begin{proof}[Proof of Claim~1]
The invariance of $Q_{\cH}\, dz^2$ under conformal changes of
coordinates is clear. Besides, it follows from the discussion above
that $\cL(g)$ and $\cM(g)$ take finite values at points $p\in
\Sigma$ where $g(p)\neq \8$. So, we only need to check that
$Q_{\cH}$ can be defined even when $g(p)=\8$. To do so first
observe that, since the potential function $R(H,q)$ satisfies that
$R(H,q)/|q|^4$ has a smooth extension to $q= \8$ for every $H\in
\R$, then $|q|^4 \cM(q)$ also has a smooth extension to $q= \8$.
Also, from the definition of $\cL$ in \eqref{defL} and the fact that
$G_z\neq 0$ (see the footnote in the statement of Theorem \ref{rep}
for the meaning of this condition at points where $G=\8$), we can
deduce that $|q|^4 \cL(q)$ also has a finite limit as $q\to \8$.
From here, one can easily show that $Q_{\cH}\, dz^2$ is well defined
even at points $p\in \Sigma$ where the Gauss map $g$ of $\psi$
satisfies $g(p)=\8$.
\end{proof}

{\bf Claim 2:} \emph{$Q_{\cH}\, dz^2$ vanishes identically on $S$.
And conversely, if $Q_{\cH}\, dz^2$ vanishes identically for some
surface $\psi:\Sigma\flecha X$ of prescribed mean curvature $\cH$,
then $\psi(\Sigma)$ is a left translation of an open subset of $S$.}
\begin{proof}[Proof of Claim~2]
The first assertion is trivial by the very definition of $Q_{\cH}$.
The converse statement can be proved using the idea in \cite[Lemma
4.6]{DM}. We include a proof here for the sake of completeness.

Let $\psi:\Sigma\flecha X$ be a conformal immersion with
$Q_{\cH}\equiv 0$, and let $g:\Sigma\flecha \bar{\C}$ be its Gauss
map. Note that $g_z\neq 0$ at every point. Also, as $G:S\equiv
\bar{\C}\flecha \bar{\C}$ is an orientation preserving
diffeomorphism, we have $|G_z|^2 - |G_{\bar{z}}|^2 >0.$ Define $\phi
: G^{-1}\circ g:\Sigma\flecha S\equiv \bar{\C}$. An elementary
computation shows that

\begin{equation}\label{eqfi}
\phi_{\bar{z}} = \frac{1}{|G_z|^2-|G_{\bar{z}}|^2} \left(
\overline{G_z} g_{\bar{z}} - G_{\bar{z}} \overline{g_z}\right),
\end{equation}
where $G_z,G_{\bar{z}}$ are evaluated at $\phi(z)$ for every $z\in
\Sigma$. Since by hypothesis $$\cL(g) g_z + \cM(g) \bar{g}_z =0,$$
we can rewrite \eqref{eqfi} as
 \begin{equation}\label{eqfi2}
\phi_{\bar{z}} = \frac{1}{|G_z|^2-|G_{\bar{z}}|^2}
\left(-\frac{\overline{\cL}(g)}{\overline{\cM}(g)} \overline{G_z} -
G_{\bar{z}} \right) \overline{g_z},
 \end{equation}
where again $G_z,G_{\bar{z}}$ are evaluated at $\phi(z)$. Observe
now that $g=G\circ \phi$ and that $\cL(G) G_z + \cM(G)
\overline{G}_z =0$ since $Q_{\cH}\equiv 0$ on $S$. This implies by
\eqref{eqfi2} that $\phi_{\bar{z}}=0$, i.e. $\phi$ is holomorphic.

Therefore, up to a local conformal change of coordinates in $\Sigma$
we can assume that $G=g$ on a neighborhood $U\subset\bar{\C}$ of an
arbitrary point $z_0$ of $\Sigma$. In particular the mean curvatures
of $\psi$ and $S$ coincide on $U$ since $\psi$ and $S$ have the same
prescribed mean curvature function $\cH$ and the same Gauss map. So,
by the uniqueness in Theorem \ref{rep} (see Remark 4 in Subsection
\ref{subrem}), we have that $\psi(U)$ differs from an open set of
$S$ by a left translation in $X$. A simple continuation argument
shows that the same is true for $\psi(\Sigma)$.
\end{proof}

{\bf Claim 3:} The maps $\cL,\cM$ defined in \eqref{defL},
\eqref{defM} satisfy the following PDEs on $\bar{\C}$:

\begin{equation}\label{eq4M}
 \cM_{\bar{q}} + (\overline{\cA} +
\cB) \cM = \cH_{\bar{q}} (1+|q|^2)^2 |\cM|^2.
 \end{equation}

\begin{equation}\label{eq4L}
  \left(\cL_q + 2\cA \cL\right) \overline{\cL} =\left(\cL_{\bar{q}} + 2 \cB \cL + \overline{\cB} \cM - \cL \cH_{\bar{q}} \overline{\cM} (1+|q|^2)^2\right) \overline{\cM},
\end{equation}

Here $\cA,\cB:\bar{\C}\flecha \bar{\C}$ are defined as
\begin{equation}\label{coefab} \cA (q) = -\frac{\cM_q}{\cM} (q),
\hspace{1cm} \cB(q)= \left( \frac{ \overline{\cM_q}}{
\overline{\cM}} - \frac{\cM_{\bar{q}}}{\cM}  \right) (q) +
(1+|q|^2)^2\cH_{\bar{q}}(q)\overline{\cM(q)}.
 \end{equation}

\begin{proof}[Proof of Claim~3]
To start, let us consider an arbitrary conformally immersed surface
$\psi:\Sigma\flecha X$ with prescribed mean curvature $\cH$, and let
$g:\Sigma\flecha \bar{\C}$ denote its Gauss map. As $R(\cH(q),q)\neq
0$ for every $q\in \bar{\C}$, we have by \eqref{eqlanda}  that
$g_z\neq 0$ on $\Sigma$. Let $H$ be the mean curvature of $\psi$,
given by $H=\cH\circ g$. Differentiating,
$$H_{\bar{z}} = \cH_q (g) g_{\bar{z}} + \cH_{\bar{q}} (g)\overline{g_z}.$$ From here, a computation shows that the Gauss map
equation \eqref{gaussmap} for $g$ can be written in this situation
as
 \begin{equation}\label{gausspre}
g_{z\bar{z}} = \cA (g) g_z g_{\bar{z}} + \cB (g) |g_z|^2,
 \end{equation} where $\cA
(q),\cB (q): \bar{\C}\flecha \bar{\C}$ are given by \eqref{coefab}.


A direct computation shows that $\cM$ satisfies \eqref{eq4M}. In
order to check that $\cL$ satisfies \eqref{eq4L} we first observe
that
\begin{equation}\label{qcero}
\cL(G) G_z + \cM(G) \bar{G}_z =0
 \end{equation}
 on $S\equiv \bar{\C}$, since $Q_{\cH}$ vanishes identically on $S$. Differentiating \eqref{qcero} with respect to $\bar{z}$ and using \eqref{eq4M} together with the fact that $G$ satisfies \eqref{gausspre} (since $S$ has prescribed mean curvature $\cH$), we obtain
  $$\def\arraystretch{1.5}\begin{array}{lll} 0 & = & (\cL_q(G) G_{\bar{z}} + \cL_{\bar{q}}(G) \overline{G_z}) G_z +
  \cL(G) G_{z\bar{z}} \\ & & + (\cM_q(G) G_{\bar{z}} + \cM_{\bar{q}}(G) \overline{G_z}) \overline{G}_z + \cM(G) \overline{G}_{z\bar{z}}  \\ & = &
  \{\cL_q + \cA \cL\} G_zG_{\bar{z}} + \{\cL_{\bar{q}} + \cB \cL + \overline{\cB} \cM\} |G_z|^2 \\ & & - \{\cA \cM\} |G_{\bar{z}}|^2 +\{-\cB \cM + \cH_{\bar{q}} |\cM|^2 (1+|q|^2)^2\} \overline{G}_z \overline{G_z},\end{array}$$ where the functions between brackets are all evaluated at $q=G(z)$. Using now the relation \eqref{qcero} in this equation, we arrive at
$$ \{\cL_q + 2\cA \cL\} G_z G_{\bar{z}} +\{\cL_{\bar{q}} + 2 \cB \cL + \overline{\cB} \cM - \cL \cH_{\bar{q}} \overline{\cM} (1+|q|^2)^2\} G_z \overline{G_z} =0.
$$
If we use now the conjugate of \eqref{qcero} and the fact that
$G_z\neq 0$, the previous equation can be reduced to
   \begin{equation}\label{eqLQ}
  \left(\cL_q + 2\cA \cL\right) \overline{\cL} =\left(\cL_{\bar{q}} + 2 \cB \cL + \overline{\cB} \cM - \cL \cH_{\bar{q}} \overline{\cM} (1+|q|^2)^2\right) \overline{\cM},   \end{equation}
  with all functions evaluated at $q=G(z)$. Since $G$ is a diffeomorphism, we deduce that \eqref{eq4L} holds.
 \end{proof}

{\bf Claim 4:} \emph{If $Q_{\cH}\, dz^2$ does not vanish identically
on a surface $\psi:\Sigma\flecha X$ with prescribed mean curvature
$\cH$, then it only has isolated zeros of negative index.}
\begin{proof}[Proof of Claim~4] Let $\psi:\Sigma\flecha X$ be a conformal immersion with prescribed mean curvature $\cH$, and let $g:\Sigma\flecha \bar{\C}$ denote its Gauss map. Recall that $g$ satisfies the PDE \eqref{gausspre}, and that $R(\cH(q),q)\neq 0$ for every $q\in \bar{\C}$. So, in particular, $g_z\neq 0$ at all points of $\Sigma$.

Let $Q_{\cH}\, dz^2$ denote the complex quadratic differential
defined in \eqref{hopfdif}. By differentiating $Q_{\cH}$ with
respect to $\bar{z}$ and using \eqref{gausspre}, we arrive at
\begin{equation}\def\arraystretch{1.5}\begin{array}{lll}\label{eqinQ1}
(Q_{\cH})_{\bar{z}} & =&  \{ \cL_q + 2 \cA \cL\} g_{\bar{z}} g_z^2 +
\{\cL_{\bar{q}} + 2\cB \cL + \overline{\cB}\cM\} g_z |g_z|^2 \\ & &
+ \{\cM_q +\cA\cM\} g_z |g_{\bar{z}}|^2 + \{\cM_{\bar{q}} +
(\overline{\cA} + B) \cM\} \overline{g}_z |g_z|^2,
\end{array}
\end{equation}
where the quantities in brackets are evaluated at $g(z)$ for every
$z\in \Sigma$. As $\cM_q + \cA \cM=0$ by definition of $\cA$, and
$\cM$ satisfies \eqref{eq4M}, we obtain from \eqref{eqinQ1}

\begin{equation}\def\arraystretch{1.7}\begin{array}{lll}\label{eqinQ2}
(Q_{\cH})_{\bar{z}} & =&  \{ \cL_q + 2 \cA \cL\} g_{\bar{z}} g_z^2 + \{\cL_{\bar{q}} + 2\cB \cL + \overline{\cB}\cM\} g_z |g_z|^2  +  \{\cH_{\bar{q}} (1+|q|^2)^2|\cM|^2\} \overline{g}_z |g_z|^2 \\ & & + \{\cL \cH_{\bar{q}} \overline{\cM} (1+|q|^2)^2\} g_z |g_z|^2 - \{\cL \cH_{\bar{q}} \overline{\cM} (1+|q|^2)^2\} g_z |g_z|^2 \\
& = & g_z^2 \left( \{ \cL_q + 2 \cA \cL\} g_{\bar{z}} +
\{\cL_{\bar{q}} + 2\cB \cL + \overline{\cB}\cM- \cL \cH_{\bar{q}}
\overline{\cM} (1+|q|^2)^2\} \overline{g_z}\right) \\ & & + |g_z|^2
(1+|g|^2)^2 \left(\{H_{\bar{q}}|\cM|^2\} \bar{g}_z +
\{\cL\cH_{\bar{q}} \overline{\cM} \} g_z\right)
\end{array}
\end{equation}
where again the quantities in brackets are evaluated at $g(z)$.
Using finally that $\cL$ satisfies \eqref{eq4L} and the definition
of $Q_{\cH}$ we obtain from \eqref{eqinQ2}

\begin{equation}\label{eqinQ3}
(Q_{\cH})_{\bar{z}} = \alfa \, Q_{\cH} + \beta \,
\overline{Q_{\cH}},
\end{equation}
where $\alfa,\beta:\Sigma\flecha \overline{\C}$ are given by
$$\alfa= \cH_{\bar{q}} (g) \overline{\cM}(g) (1+|g|^2)^2 \overline{g_z}, \hspace{0.5cm} \beta= \left(\frac{\cL_q + 2\cA \cL}{\overline{\cM}}\right)(g) \frac{g_z^2}{\overline{g_z}}.$$ Since $g_z\neq 0$ on $\Sigma$, it is clear that $\alfa(z_0),\beta(z_0)$ take values in $\C$ (i.e. they are finite) whenever $g(z_0)\neq \8$. So, from \eqref{eqinQ3} and $Q_{\cH}\not\equiv 0$ we have that
 \begin{equation}\label{eqinQ4}
\frac{|(Q_{\cH})_{\bar{z}}|}{|Q_{\cH}|} \hspace{0.4cm} \text{is
locally bounded}
\end{equation}
around every $z_0\in \Sigma$ with $g(z_0)\neq \8$.

When $g(z_0)=\8$ one can easily show from the above formulas that
\eqref{eqinQ4} also holds, by considering  the map $\xi=1/g$ around
$z_0$. Thus, \eqref{eqinQ4} holds at all points. This condition is
well known to imply that $Q_{\cH}$ only has isolated zeros of
negative index (see \cite{ADT,Jo} for instance).

\end{proof}

Note that Claims 1, 2 and 4 prove Theorem \ref{hopf}.

\end{proof}

Let us point out here that the proof of Theorem \ref{hopf} above also holds without the assumption that the surface $S$ is compact. Specifically, we have:

\begin{corollary}
Let $S$ be an immersed surface with prescribed mean curvature $\cH\in C^1(\S^2)$ in a metric Lie group $X$, and assume that its Gauss map $G:S\flecha \cU:=G(S)\subset \S^2$ is an orientation preserving diffeomorphism onto its image.

Then there exists a complex quadratic differential $Q_{\cH} \, dz^2$ defined for any immersed surface $\psi:\Sigma\flecha X$ with prescribed mean curvature $\cH$ and Gauss map image contained in $\cU$, so that conditions (1), (2) and (3) in Theorem \ref{hopf} hold.
\end{corollary}

\section{Final remarks}\label{sec5}

\subsection{Necessity of strict convexity}

Theorem \ref{mainr3} is not true in general if we do not assume that there exists a strictly convex sphere with prescribed mean curvature $\cH\in C^1(\S^2)$. In fact, the condition of $S$ being \emph{strictly} convex (i.e. of positive curvature at every point) cannot be weakened to $S$ being just convex (i.e. of non-negative curvature), as the following example shows.

\begin{example}
Let $C=\{(x,y,z)\in \R^3: x^2+y^2= 1, -1\leq z \leq 1\}$, and let $\mathcal{G}$ be a smooth rotational convex graph $z=z(x,y)$ on $\overline{D}=\{(x,y): x^2+y^2\leq 1\}$ with $z(\parc D) \equiv 1$, so that $S=C\cup \mathcal{G} \cup (-\mathcal{G})$ is a convex (but not strictly convex) sphere in $\R^3$.

Now let $C'=\{(x,y,z)\in \R^3 :x^2 +y^2 =1, -2\leq z \leq 2\}$, let $\mathcal{G}'$ be the vertical translation of $\mathcal{G}$ so that its horizontal boundary is contained in the plane $z=2$, and define $S'= C'\cup \mathcal{G}' \cup (-\mathcal{G}')$, which is again a (not strictly) convex sphere in $\R^3$.

It is then clear that there is a diffeomorphism $\phi:S\flecha S'$ such that, for every $p\in S$:
 \begin{enumerate}
 \item
The Gauss map of $S$ at $p$ agrees with the Gauss map of $S'$ at $\phi(p)$.
 \item
The principal curvatures of $S$ at $p$ agree with the principal curvatures of $S'$ at $\phi(p)$.
 \end{enumerate}
 In particular, $S$ and $S'$ are two convex spheres in $\R^3$ with the same prescribed mean curvature but which do not coincide up to translation in $\R^3$.
\end{example}

\subsection{Higher order contact with spheres}

Let $S,S^*$  be two immersed surfaces of prescribed mean curvature $\cH\in C^1(\S^2)$ in a metric Lie group $X$, and assume that they have a contact point of order $k\geq 1$ at $p\in S\cap S^*$. We assume that the potential function $R$ satisfies $R(\cH(q_0),q_0)\neq 0$ where $q_0\in \bar{\C}$ denotes the common Gauss map image of both $S$, $S^*$ at $p$.  By reparametrizing both surfaces in a suitable way we may view $S,S^*$ around $p$ as two conformal immersion $\psi,\psi^*:\D\flecha X$ with $\psi(0)=\psi^*(0)=p$, whose Gauss maps $g,g^*$ satisfy $g(0)=g^*(0)=: q_0\in \bar{\C}$, and such that $g_z(0)=(g^*)_z(0)=1$. Note that these conditions imply by \eqref{eqlanda} that the conformal factors $\landa,\landa^*$ verify $\landa(0)=\landa^*(0)$.

Also, note that the mean curvatures of $\psi$ and $\psi^*$ also coincide at $0$. Thus, $S,S^*$ have a contact point of order $k\geq 2$ at $p$ if and only if their respective Hopf differentials $Pdz^2$, $P^* dz^2$ satisfy $P(0)=P^*(0)$. It follows from the expressions of $\landa$ and $P$ in \eqref{eqlanda} and \eqref{calho2} that, in our conditions, $P(0)=P^*(0)$ is equivalent to $\bar{g}_z(0)=(\overline{g^*})_z(0)$.

Suppose now that $S$ is a compact surface of prescribed mean curvature $\cH$ whose Gauss map is a diffeomorphism into $\bar{\C}$. Note that in this case the condition $R(\cH(q_0),q_0)\neq 0$ holds automatically for every $q_0\in \bar{\C}$, see Section 4. Let $Q_{\cH}\, dz^2$ denote the complex quadratic differential associated to $S$, given by \eqref{hopfdif}; note that $Q_{\cH}\, dz^2$ is defined for any conformally immersed surface in $X$ with prescribed mean curvature $\cH$.

Then, using the previous discussion together with the fact that $Q_{\cH}\equiv 0$ on $S$, it is easy to check that a surface $S^*$ of prescribed mean curvature $\cH$ in $X$ has a point $p\in S^*$ with $Q_{\cH}(p)=0$ if and only if $S^*$  has a contact of order $k\geq 2$ at $p$ with a left translation of the sphere $S$.

In addition, as $Q_{\cH} \, dz^2$ only has isolated zeros of negative index on $S^*$ by Theorem \ref{hopf}, the Poincaré-Hopf theorem shows that if $S^*$ is compact and of genus $g\geq 1$, then the number of zeros of $Q_{\cH}\, dz^2$ on $S^*$ (counted with multiplicities) is finite and equal to $4g-4$.

As a consequence of this discussion and the existence of the Guan-Guan spheres in $\R^3$ for $\cH\in C^2(\S^2)$ with $\cH(x)=\cH(-x)>0$, we have the next corollary:

\begin{corollary}
Let $\cH\in C^2(\S^2)$ satisfy $\cH(x)=\cH(-x)>0$ for every $x\in \S^2$, let $S_{\cH}\subset \R^3$ be the Guan-Guan sphere for $\cH$, and let $\Sigma$ be a compact immersed surface in $\R^3$ of genus $g$ with prescribed mean curvature $\cH$. Then, up to a translation in $\R^3$:
\begin{enumerate}
\item
If $g=0$, then $\Sigma= S_{\cH}$.
 \item
If $g=1$, then $\Sigma$ has no point of contact with $S_{\cH}$ of order greater than one.
 \item
If $g\geq 2$, then $\Sigma$ has at most $4g-4$ points of contact with $S_{\cH}$ of order greater than one.
\end{enumerate}
\end{corollary}

\def\refname{References}

\end{document}